\documentclass[12pt]{article}
\usepackage{amsmath,amsfonts,amssymb,amsthm,txfonts,graphicx,hyperref}
\usepackage[a4paper,margin=2.5cm]{geometry}
\newtheorem{theorem}{Theorem}
\newtheorem{lemma}{Lemma}
\newcommand{\N}{\mathbb N}
\newcommand{\Z}{\mathbb Z}
\newcommand{\Q}{\mathbb Q}
\title{Hilbert's 10th Problem for solutions in a subring of $\Q$}
\author{Agnieszka Peszek, Apoloniusz Tyszka}
\begin{document}
\date{}
\maketitle
\begin{sloppypar}
\begin{abstract}
Yuri Matiyasevich's theorem states that the set of all Diophantine equations which
have a solution in \mbox{non-negative} integers is not recursive.
Craig Smory\'nski's theorem states that the set of all Diophantine equations which
have at most finitely many solutions in \mbox{non-negative} integers is not recursively enumerable.
Let $R$ be a subring of~$\Q$ with or without~$1$. By \mbox{$H_{10}(R)$}, we denote
the problem of whether there exists an algorithm which for any given Diophantine equation with integer coefficients,
can decide whether or not the equation has a solution in $R$.
We prove that a positive solution to \mbox{$H_{10}(R)$} implies that the set of all Diophantine equations
with a finite number of solutions in $R$ is recursively enumerable. We show the converse implication
for every infinite set \mbox{$R \subseteq \Q$} such that there exist computable functions
\mbox{$\tau_1,\tau_2 \colon \N \to \Z$} which satisfy \mbox{$(\forall n \in \N$} $\tau_2(n) \neq 0) \wedge
\left(\left\{\frac{\textstyle \tau_1(n)}{\textstyle \tau_2(n)}:~n \in \N\right\}=R\right)$.
This implication for \mbox{$R=\N$} guarantees that Smory\'nski's theorem follows from Matiyasevich's theorem.
Harvey Friedman conjectures that the set of all polynomials of several variables with integer
coefficients that have a rational solution is not recursive.
Harvey Friedman conjectures that the set of all polynomials of several variables with integer coefficients
that have only finitely many rational solutions is not recursively enumerable.
These conjectures are equivalent by our results for \mbox{$R=\Q$.}
\end{abstract}
\vskip 0.2truecm
\noindent
{\bf Key words and phrases:} Craig Smory\'nski's theorem, Diophantine equation which has at most finitely many solutions,
Hilbert's 10th Problem for solutions in a subring of $\Q$, Martin Davis' theorem, recursive set, recursively enumerable set,
Yuri Matiyasevich's theorem.
\section{Introduction and basic lemmas}
Yuri Matiyasevich's theorem states that the set of all Diophantine equations which
have a solution in \mbox{non-negative} integers is not recursive, see \cite{Matiyasevich}.
Martin Davis' theorem states that the set of all Diophantine equations which have at most finitely many solutions
in positive integers is not recursive, see \cite{Davis}. Craig Smory\'nski's theorem states that
the set of all Diophantine equations which have at most finitely many solutions in \mbox{non-negative} integers is not
recursively enumerable, see \mbox{\cite[p.~104,~Corollary~1]{Smorynski1}} and \mbox{\cite[p.~240]{Smorynski2}}.
\vskip 0.2truecm
\par
Let ${\cal P}$ denote the set of prime numbers, and let
\[
{\cal P}=\{p_1,q_1,r_1,p_2,q_2,r_2,p_3,q_3,r_3,\ldots\},
\]
where \mbox{$p_1<q_1<r_1<p_2<q_2<r_2<p_3<q_3<r_3<\ldots$} 
\begin{lemma}\label{lem0}
For a \mbox{non-negative} integer $x$, let
$\prod_{i=1}^\infty\limits p_i^{\textstyle \alpha_i} \cdot q_i^{\textstyle \beta_i} \cdot r_i^{\textstyle \gamma_i}$
be the prime decomposition of \mbox{$x+1$}. For every positive integer $n$, the mapping which sends \mbox{$x \in \N$} to
\[
\left((-1)^{\textstyle \alpha_1} \cdot \frac{\textstyle \beta_1}{\textstyle \gamma_1+1},\ldots,(-1)^{\textstyle \alpha_n} \cdot \frac{\textstyle \beta_n}{\textstyle \gamma_n+1}\right) \in {\Q}^n
\]
is a computable surjection from $\N$ onto \mbox{${\Q}^n$}.
\end{lemma}
\vskip 0.2truecm
\par
Let \mbox{$s_n \colon \N \to {\Q}^n$} denote the surjection defined in Lemma~\ref{lem0}.
\noindent
\begin{lemma}\label{lem1}
For every infinite set \mbox{$R \subseteq \Q$},
a Diophantine equation \mbox{$D(x_1,\ldots,x_n)=0$} has no solutions
in \mbox{$x_1,\ldots,x_n \in R$} if and only if the equation \mbox{$D(x_1,\ldots,x_n)+0 \cdot x_{n+1}=0$} has at most
finitely many solutions in \mbox{$x_1,\ldots,x_{n+1} \in R$}.
\end{lemma}
\par
Let $R$ be a subring of~$\Q$ with or without $1$. By \mbox{$H_{10}(R)$}, we denote
the problem of whether there exists an algorithm which for any given Diophantine equation with integer coefficients,
can decide whether or not the equation has a solution in $R$.
\section{A positive solution to \mbox{$H_{10}(R)$} implies that the set of all
Diophantine equations with a finite number of solutions in $R$ is recursively enumerable}
In the next three lemmas we assume that \mbox{$\{0\} \subsetneq R \subseteq \Q$}
and \mbox{$r \cdot \Z \subseteq R$} for every \mbox{$r \in R$}. Every \mbox{non-zero}
subring $R$ of $\Q$ (with or without $1$) satisfies these conditions.
\begin{lemma}\label{April2019}
There exists a \mbox{non-zero} integer \mbox{$m \in R$}.
\end{lemma}
\begin{proof}
There exist \mbox{$m,n \in \Z \setminus \{0\}$} such that \mbox{$\frac{\textstyle m}{\textstyle n} \in R$}.
Hence,
\mbox{$m=\frac{\textstyle m}{\textstyle n} \cdot n \in (\Z \setminus \{0\}) \cap R$}.
\end{proof}
\begin{lemma}\label{4sq}
Let \mbox{$m \in (\Z \setminus \{0\}) \cap R$}. We claim that for every \mbox{$b \in R$},
\mbox{$b \neq 0$} if and only if the equation
\[
y \cdot b-m^2-\sum_{i=1}^4 y_i^2=0
\]
is solvable in \mbox{$y,y_1,y_2,y_3,y_4 \in R$}.
\end{lemma}
\begin{proof}
If \mbox{$b=0$}, then for every \mbox{$y,y_1,y_2,y_3,y_4 \in R$},
\[
y \cdot b-m^2-y_1^2-y_2^2-y_3^2-y_4^2=-m^2-y_1^2-y_2^2-y_3^2-y_4^2 \leqslant -m^2<0
\]
If \mbox{$b \neq 0$}, then \mbox{$b=\frac{\textstyle p}{\textstyle q}$}, where \mbox{$p \in \N \setminus \{0\}$}
and \mbox{$q \in \Z \setminus \{0\}$}. In this case, we define $y$ as \mbox{$m^2 \cdot q$}
and observe that \mbox{$m^2 \cdot q=(m \cdot q) \cdot m \in R$} as \mbox{$m \cdot q \in R$} and \mbox{$m \in \Z$}.
Hence,
\[
y \cdot b=(m^2 \cdot q) \cdot \frac{\textstyle p}{\textstyle q}=m^2 \cdot p \in m^2 \cdot (\N \setminus \{0\})
\]
By Lagrange's \mbox{four-square} theorem, there exist \mbox{$t_1,t_2,t_3,t_4 \in \N$} such that
\[
\frac{y \cdot b-m^2}{m^2}=t_1^2+t_2^2+t_3^2+t_4^2
\]
Therefore,
\[
y \cdot b-m^2-(m \cdot t_1)^2-(m \cdot t_2)^2-(m \cdot t_3)^2-(m \cdot t_4)^2=0,
\]
where \mbox{$m \cdot t_1,~m \cdot t_2,~m \cdot t_3,~m \cdot t_4 \in R$}.
\end{proof}
\newpage
\begin{lemma}\label{trivial}
We can uniquely express every rational number $r$ as \mbox{$\widehat{~~r~~}/~\overline{r}$},
where \mbox{$\widehat{~~r~~} \in \Z$}, \mbox{$\overline{r} \in \N \setminus \{0\}$},
and the integers $\widehat{~~r~~}$ and $\overline{r}$ are relatively prime. If \mbox{$r \in R$},
then \mbox{$\widehat{~~r~~} \in R$}.
\end{lemma}
\begin{proof}
For every \mbox{$r \in R$}, \mbox{$\widehat{~~r~~}=r \cdot \overline{r} \in r \cdot \Z \subseteq R$}.
\end{proof}
\begin{lemma}\label{apr23}
Let $R$ be a \mbox{non-zero} subring of $\Q$ with or without $1$.
We claim that for every \mbox{$T_0,\ldots,T_k \in R^n$} and for every \mbox{$x_1,\ldots,x_n \in R$},
the following product
\begin{equation}\label{equ1}
\prod_{\textstyle \left(r_{1},\ldots,r_{n}\right) \in \left\{T_0,\ldots,T_k\right\}}\ 
\sum_{i=1}^{n} \left(x_{i} \cdot \overline{r_i} -\widehat{~~r_i~~}\right)^{2}
\end{equation}
differs from $0$ if and only if \mbox{$(x_1,\ldots,x_n) \not\in \{T_0,\ldots,T_k\}$}.
Product~(\ref{equ1}) belongs to $R$.
\end{lemma}
\begin{proof}
The last claim follows from Lemma~\ref{trivial}.
\end{proof}
\begin{lemma}\label{3maj}
Let $R$ be a \mbox{non-zero} subring of $\Q$ (with or without $1$) such that
there exists an algorithm which for every \mbox{$(a,b) \in \Z \times (\Z \setminus \{0\})$}
decides whether or not \mbox{$\frac{\textstyle a}{\textstyle b} \in R$}.
Let \mbox{$\rho_n \colon {\Q}^n \to R^n$} denote the function
which equals the identity on \mbox{$R^n$} and equals \mbox{$(0,\ldots,0)$} outside \mbox{$R^n$}.
We claim that for every positive integer $n$ the function \mbox{$\rho_n \circ s_n \colon \N \to R^n$}
is surjective and computable.
\end{lemma}
\begin{theorem}\label{the6}
Let $R$ be a \mbox{non-zero} subring of $\Q$ (with or without $1$) such that
Hilbert's 10th Problem for solutions in $R$ has a positive solution. We claim
that the set of all Diophantine equations with a finite number of solutions
in~$R$ is recursively enumerable.
\end{theorem}
\begin{proof}
By Lemma~\ref{April2019}, there exists a \mbox{non-zero} integer \mbox{$m \in R$}.
For every \mbox{$(a,b) \in \Z \times (\Z \setminus \{0\})$}, the solvability in $R$
of the equation \mbox{$b \cdot x-a=0$} is decidable. Hence,
for every \mbox{$(a,b) \in \Z \times (\Z \setminus \{0\})$} we can decide whether
or not \mbox{$\frac{\textstyle a}{\textstyle b} \in R$}.
By \mbox{Lemmas~\ref{4sq} and~\ref{apr23}}, the answer to the question in Flowchart~1 is positive
if and only if the equation \mbox{$D(x_1,\ldots,x_n)=0$} is solvable in \mbox{$R^n \setminus \{\theta(0),\ldots,\theta(k)\}$}.
Hence, by Lemma~\ref{3maj}, the algorithm in Flowchart~1 halts if and only if the equation \mbox{$D(x_1,\ldots,x_n)=0$}
has at most finitely many solutions in $R$.
\begin{center}
\includegraphics[width=\textwidth]{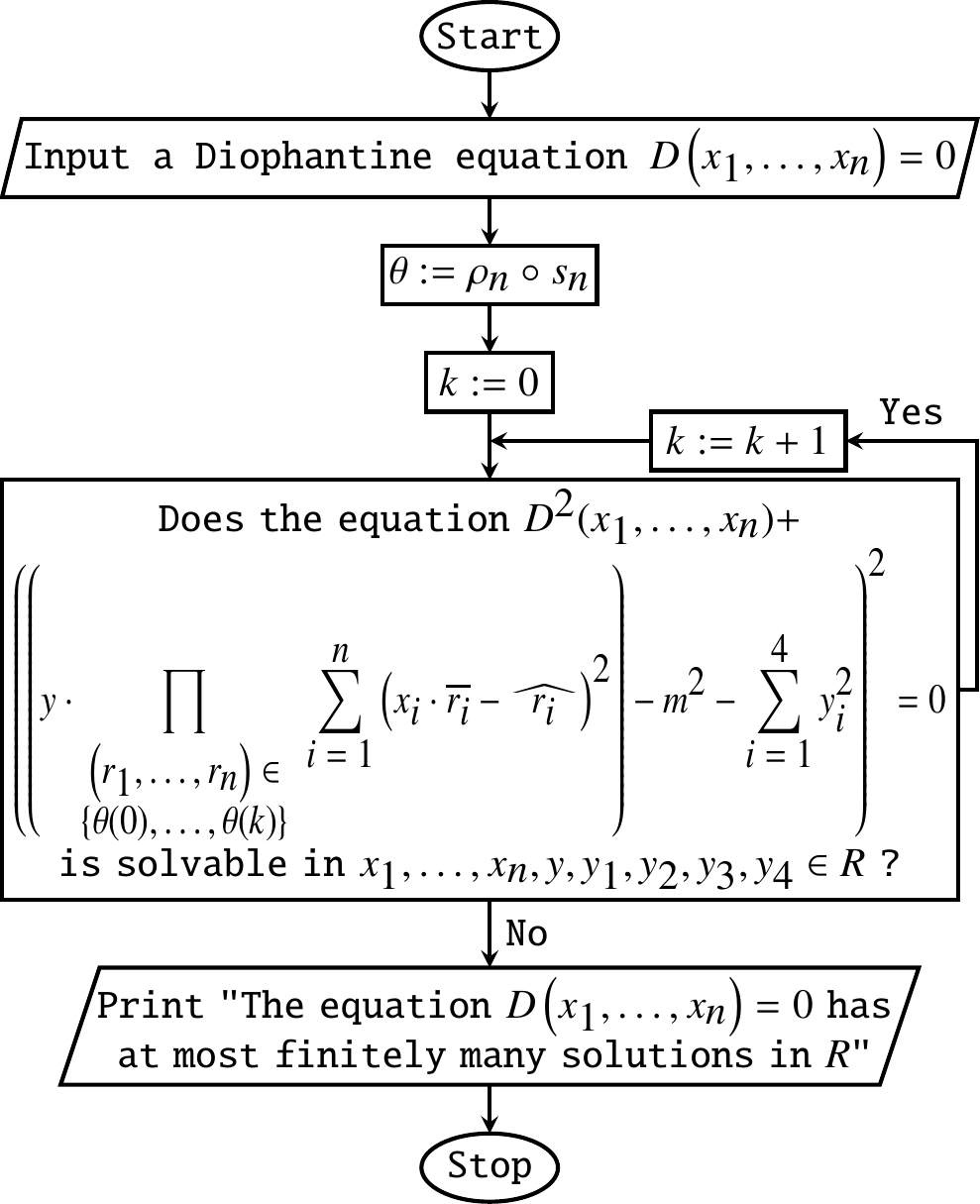}
\end{center}
\vskip 0.01truecm
\centerline{{\bf Flowchart 1}}
\end{proof}
\par
Theorem~\ref{the6} remains true when \mbox{$R=\{0\}$}.
The flowchart algorithm depends on \mbox{$m \in (\Z \setminus \{0\}) \cap R$}.
For a constructive proof of Theorem~\ref{the6}, we must compute an element of \mbox{$(\Z \setminus \{0\}) \cap R$}.
By Lemma~\ref{3maj}, the function \mbox{$\rho_n \circ s_n \colon \N \to R^n$} is computable and surjective.
We compute the smallest \mbox{$i \in \N$} such that \mbox{$(\rho_n \circ s_n)(i)$}
starts with a \mbox{non-zero} integer. This integer belongs to \mbox{$(\Z \setminus \{0\}) \cap R$}.
\section{If the set of all Diophantine equations with a finite
number of solutions in $R$ is recursively enumerable, then $H_{10}(R)$ has a positive solution}
Starting from this moment up to the end of Theorem~\ref{the5}, we assume that $R$ is an infinite subset
of~$\Q$ and there exist computable functions \mbox{$\tau_1,\tau_2 \colon \N \to \Z$} which satisfy
\[
(\forall n \in \N ~\tau_2(n) \neq 0) \wedge \left(\left\{\frac{\textstyle \tau_1(n)}{\textstyle \tau_2(n)}:~n \in \N\right\}=R\right)
\]
In other words, the function
$\N \ni n \stackrel{\textstyle \tau}{\longrightarrow} \frac{\textstyle \tau_1(n)}{\textstyle \tau_2(n)} \in R$
is surjective and computable. Hence, the function \mbox{$(\tau,\ldots,\tau): {\N}^n \to R^n$} is surjective
and computable.
\begin{lemma}\label{may6}
Let \mbox{$\sigma_n \colon {\Q}^n \to {\N}^n$} denote the function which equals the
identity on \mbox{${\N}^n$} and equals \mbox{$(0,\ldots,0)$} outside \mbox{${\N}^n$}.
We claim that for every positive integer $n$ the function
\mbox{$(\tau,\ldots,\tau) \circ \sigma_n \circ s_n \colon \N \to R^n$}
is surjective and computable.
\end{lemma}
\begin{theorem}\label{the5}
If the set of all Diophantine equations which have at most finitely many solutions in $R$
is recursively enumerable, then there exists an algorithm which decides whether or not a given Diophantine
equation has a solution in $R$.
\end{theorem}
\begin{proof}
Suppose that \mbox{$\{{\cal S}_i=0\}_{i=0}^\infty$} is a computable sequence of all Diophantine equations which have
at most finitely many solutions in~$R$. By Lemma~\ref{lem1},
the execution of Flowchart 2 decides whether or not a Diophantine equation \mbox{$D(x_1,\ldots,x_n)=0$}
has a solution in~$R$.
The flowchart algorithm uses a computable surjection \mbox{$\varphi \colon \N \to R^n$}
(which exists by Lemma~\ref{may6}).
\begin{center}
\includegraphics[width=\textwidth]{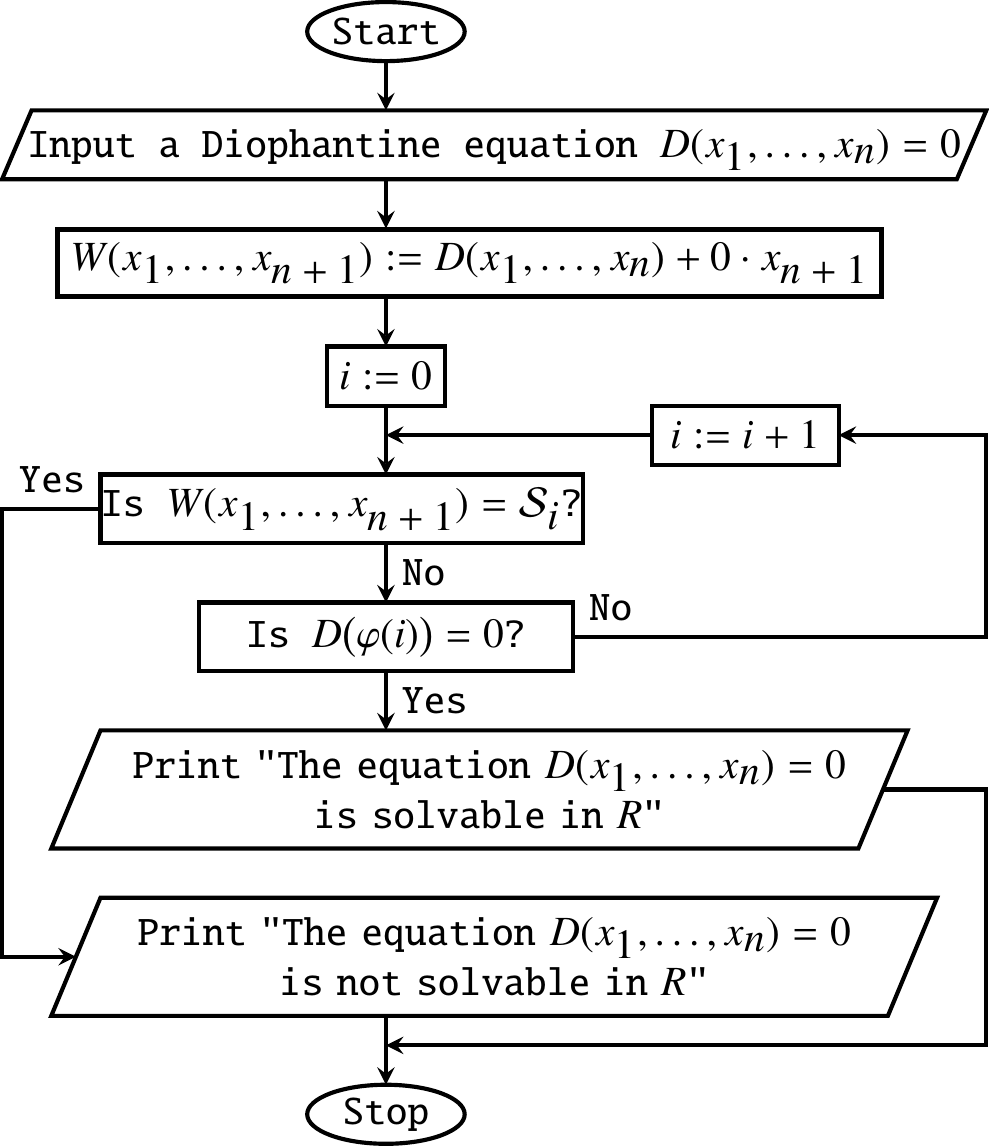}
\end{center}
\vskip 0.01truecm
\centerline{{\bf Flowchart 2}}
\vskip 0.2truecm
\noindent
The flowchart algorithm always terminates because there exists a \mbox{non-negative} integer $i$
such that
\[
(D(x_1,\ldots,x_n)+0 \cdot x_{n+1}={\cal S}_i) \vee (D(\varphi(i))=0)
\]
Indeed, for every Diophantine equation \mbox{$D(x_1,\ldots,x_n)=0$},
the flowchart algorithm finds a solution in $R$,
or finds the equation \mbox{$D(x_1,\ldots,x_n)+0 \cdot x_{n+1}=0$} on the infinite
list \mbox{$[{\cal S}_0,{\cal S}_1,{\cal S}_2,\ldots]$} if the equation \mbox{$D(x_1,\ldots,x_n)=0$}
is not solvable in $R$.
\end{proof}
\vskip 0.2truecm
\noindent
{\bf Corollary.} {\em Theorem~\ref{the5} for \mbox{$R=\N$} implies that
Craig Smory\'nski's theorem follows from Yuri Matiyasevich's theorem.}
\newpage
Harvey Friedman conjectures that the set of all polynomials of several variables with integer
coefficients that have a rational solution is not recursive, see \cite{Friedman}.
Harvey Friedman conjectures that the set of all polynomials of several variables with integer coefficients
that have only finitely many rational solutions is not recursively enumerable, see \cite{Friedman}.
These conjectures are equivalent by Theorems~\ref{the6} and~\ref{the5} taking \mbox{$R=\Q$}.
\vskip 0.2truecm
\noindent
{\bf Acknowledgement.} Agnieszka Peszek prepared two flowcharts in {\sl TikZ}. Apoloniusz Tyszka wrote the article.

\vskip 0.01truecm
\noindent
Agnieszka Peszek\\
University of Agriculture\\
Faculty of Production and Power Engineering\\
Balicka 116B, 30-149 Krak\'ow, Poland\\
E-mail: \url{Agnieszka.Peszek@urk.edu.pl}
\vskip 0.2truecm
\noindent
Apoloniusz Tyszka\\
University of Agriculture\\
Faculty of Production and Power Engineering\\
Balicka 116B, 30-149 Krak\'ow, Poland\\
E-mail: \url{rttyszka@cyf-kr.edu.pl}
\end{sloppypar}
\end{document}